\documentclass{amsart}
\usepackage{graphicx} 

\usepackage{amsfonts}
\usepackage{amsmath}
\usepackage{amsthm}
\usepackage{amssymb}
\usepackage{amsrefs}
\usepackage{geometry}
\usepackage{comment}
\usepackage{mathrsfs}
\usepackage{xcolor}

\BibSpec{inbook}{%
    +{}  {\PrintAuthors}                {author}
    +{,} { \textit}                     {title}
    +{.} { }                            {part}
    +{:} { \textit}                     {subtitle}
    +{,} { \PrintContributions}         {contribution}
    +{,} { \PrintConference}            {conference}
    +{}  {\PrintBook}                   {book}
    +{,} { }                            {booktitle}
    +{,} { \PrintEditorsB}              {editor}
    +{,} { }                            {publisher}
    +{,} { \PrintDateB}                {date}
    +{,} {~}                        {pages}
    +{,} { }                            {status}
    +{,} { \PrintDOI}                   {doi}
    +{,} { available at \eprint}        {eprint}
    +{}  { \parenthesize}               {language}
    +{}  { \PrintTranslation}           {translation}
    +{;} { \PrintReprint}               {reprint}
    +{.} { }                            {note}
    +{.} {}                             {transition}
    +{}  {\SentenceSpace \PrintReviews} {review}
}

\renewcommand{\epsilon}{\varepsilon}
\renewcommand{\subset}{\subseteq}
\renewcommand{\supset}{\supseteq}

\usepackage{titlesec}

\theoremstyle{definition}
\newtheorem{definition}{Definition}[section]
\newtheorem{theorem}[definition]{Theorem}
\newtheorem{lemma}[definition]{Lemma}
\newtheorem{corollary}[definition]{Corollary}
\newtheorem{example}[definition]{Example}
\newtheorem{remark}[definition]{Remark}
\newtheorem{proposition}[definition]{Proposition}

\titleformat{\section}{\normalfont\bfseries}
{\thesection.}{0.5em}{}

\newcommand\blfootnote[1]{%
  \begingroup
  \renewcommand\thefootnote{}\footnote{#1}%
  \addtocounter{footnote}{-1}%
  \endgroup
}

\title{\textsc{The Normality of Products Under Perfect Preimages}}
\author{\textsc{Lucas D. O'Brien}}
\date{}

\begin{document}

\maketitle

\begin{abstract}
    A proof of the following theorem is given, answering an open problem attributed to Kunen: suppose that $T$ is compact and that $Y$ is the image of $X$ under a perfect map, $X$ is normal, and $Y\times T$ is normal. Then $X \times T$ is normal.
\end{abstract}



\section{Introduction}

In \cite{Rudin73}, Mary Ellen Rudin proved that given a compact space $T$, a normal space $X$, and a space $Y$ that is the image of $X$ under a closed map, that $X \times T$ is normal only if $Y \times T$ is normal. Rudin's proof relied on a characterization of normality of products with a compact factor which appeared as Theorem 3 in \cite{Rudin73}, and is included as Theorem \ref{characterizationtheorem} in this paper. This characterization loosely says that the normality of the product $X \times T$ is equivalent to certain open covers of $X$ having locally finite open refinements. Rudin called these open covers $X$-\textit{separations}; we follow the convention of \cite{Kunen84} and call them $\mathscr{B}$-\textit{coverings} (see Definition \ref{covering}). In view of Theorem \ref{characterizationtheorem}, there is correlation between paracompactness and the normality of a product with a compact factor. Given a compact space $T$, Morita defines a space $X$ to be $T$-\textit{paracompact} provided $X \times T$ is normal \cite{Morita61}. 

We may intuitively expect the property of $T$-paracompactness to be preserved in a similar way to paracompactness under mappings. Indeed, Rudin's theorem shows that $T$-paracompactness is preserved under images of closed maps. The question of whether $T$-paracompactness is preserved under preimages of perfect maps is attributed to Kunen and appeared as an open problem in \cite{Kunen84} and \cite{Rudin75}; this paper shows that this is indeed the case. 

By the Tietze extension theorem, any continuous map from a closed subspace of a normal space $X$ to $\mathbb{R}$ can be extended to a continuous map of all of $X$ into $\mathbb{R}$. Therefore, the study of the normality of products has applications to the problem of the extendability of continuous functions from subsets of product spaces. For more details, see \cite{Reed80}.

\section{ Characterization of Normality in Products with a Compact Factor}

In the following, $T$ denotes a compact Hausdorff space and $\mathscr{B}$ denotes a base of $T$ which is closed with respect to finite unions and intersections. A \textit{normal space} refers to a topological space which is both $\text{T}_4$ and $\text{T}_2$. 

\begin{definition}
    Let $\mathcal{A}$ be a collection of sets, and let $\mathcal{B}$ be a collection of ordered pairs of sets. We will say a map $G: \mathcal{B} \to \mathcal{A}$ is \textit{intersection preserving} provided $G(A,B) \cap G(A',B') = G(A \cap A', B\cap B')$ for all $(A,B), (A',B') \in \mathcal{B}$. 
\end{definition}

\begin{remark}\label{intersectionpropertyremark}
    Suppose $G: \mathcal{B} \to \mathcal{A}$ is intersection preserving. Then, for $(A,B), (A',B') \in \mathcal{B}$ such that $A \subset A'$ and $B \subset B'$, we have $G(A,B) \subset G(A',B')$. 
\end{remark}

\begin{definition}
    \cite{Kunen84}\label{covering} Fix a base $\mathscr{B}$ of $T$, and define $\mathcal{S}_{\mathscr{B}} = \{(B,D) \in \mathscr{B}^2| \  \overline{B} \cap \overline{D} = \emptyset\}$. An open cover $\mathscr{G}$ of a space $X$ is called a $\mathscr{B}$-\textit{covering of} $X$ if there exists an intersection preserving surjection $G: \mathcal{S}_{\mathscr{B}} \to \mathscr{G}$. We call $G$ \textit{the associated map} to the $\mathscr{B}$-covering.
\end{definition}

Let $X$ be a space, and let $G: \mathcal{S}_{\mathscr{B}} \to \mathscr{P}(X)$ be an intersection preserving map. Then, the set $\mathscr{G} = \{G(B,D) \ | \ (B,D) \in \mathcal{S}_{\mathscr{B}}\}$ is a $\mathscr{B}$-covering of $X$ if and only if $\mathscr{G}$ is an open cover of $X$.

\begin{definition}\label{precovering}
    An open cover $\mathscr{G}'$ of a space $X$ is called a \textit{pre}-$\mathscr{B}$-\textit{covering of} $X$ if there exists an intersection preserving surjection $G : \mathscr{B}^2 \to \mathscr{G}'$ and $\mathscr{G} = \{G(B,D) \ | \  (B,D) \in \mathcal{S}_{\mathscr{B}}\} \subset \mathscr{G}'$ is a $\mathscr{B}$-covering of $X$. We call $\mathscr{G}$ the $\mathscr{B}$-\textit{covering associated with} $\mathscr{G}'$. 
\end{definition}

\begin{example} \label{coveringexample}
    Let $X$ be a normal space, and let $K,L \subset X \times T$ be closed, disjoint sets. For $B,D \in \mathscr{B}$, define

    \[
    G(B,D) = \{x \in X  \ | \  K_x \subset B \text{ and }L_x \subset D\},
    \]

    where $K_x  = \{t \in T \ | \ (x,t) \in K\}$ and $L_x =\{t \in T \ | \ (x,t) \in L\}$. Then, the set $\mathscr{G}' = \{G(B,D) \ | \ (B,D) \in \mathscr{B}^2\}$ is a pre-$\mathscr{B}$-covering of $X$.

    Indeed, by definition, $G: \mathscr{B}^2 \to \mathscr{G}, (B,D)\mapsto G(B,D)$ is surjective, and it is easy to check that $G$ is intersection preserving. Thus, it remains to show that $\mathscr{G} = \{G(B,D) \ | \  (B,D) \in \mathcal{S}_{\mathscr{B}}\}$ covers $X$ and $G(B,D)$ is open for each $(B,D) \in \mathscr{B}^2$. Let $x \in X$. Since $K$ and $L$ are disjoint, so are $K_x$ and $L_x$. Furthermore, notice that $K_x = \pi_T(K \cap (\{x\}\times T))$ and $L_x = \pi_T(L \cap (\{x\}\times T))$, where $\pi_T: X \times T \to T$ is the projection map. Since $K \cap (\{x\}\times T)$ and $L \cap (\{x\}\times T)$ are compact, $K_x$ and $L_x$ are closed. Thus, there exist $B,D \in \mathscr{B}$ such that $K_x \subset B$, $L_x \subset D$ and $\overline{B} \cap \overline{D} = \emptyset$ since $\mathscr{B}$ is closed with respect to finite unions. Then, $x \in G(B,D)$. Notice that for each $(B,D) \in \mathscr{B}^2$, we may write

    \[
    G(B,D) = X \setminus \pi_X((K\setminus (X \times B))\cup (L \setminus (X\times D))),
    \]

    where $\pi_X: X \times T \to X$ is the projection map. Since $T$ is compact, $\pi_X$ is a closed map, and therefore $G(B,D)$ is open. 
\end{example}

The next lemma follows from the techniques used in \cite{Kunen84}. 

\begin{lemma} \label{precoveringlemma}
    Let $X$ be a normal space. For every $\mathscr{B}$-covering $\mathscr{G}$ of $X$, there exists a pre-$\mathscr{B}$-covering $\mathscr{H}'$ of $X$ such that the $\mathscr{B}$-covering associated with $\mathscr{H}'$ is an open cover refinement of $\mathscr{G}$.
\end{lemma}

\begin{proof}
    Let $\mathscr{G}$ be a $\mathscr{B}$-covering of $X$, and let $G: \mathcal{S}_{\mathscr{B}} \to \mathscr{G}$ be the associated map. Define 

    \[
    K = (X \times T) \setminus \bigcup\{G(B,D) \times (T \setminus \overline{B}) \ | \ (B,D) \in \mathcal{S}_{\mathscr{B}}\}
    \]

    and 

    \[
    L = (X \times T) \setminus \bigcup\{G(B,D) \times (T \setminus \overline{D}) \ | \  (B,D) \in \mathcal{S}_{\mathscr{B}}\}.
    \]

    Evidently $K$ and $L$ are closed, we will show they are disjoint. Indeed, let $(x,t) \in X \times T$. Since $\mathscr{G}$ covers $X$, there exists $(B,D) \in \mathcal{S}_{\mathscr{B}}$ such that $x \in G(B,D)$. We must have $t \notin \overline{B}$ or $t \notin \overline{D}$, and therefore either $(x,t) \in G(B,D) \times (T \setminus \overline{B})$ or $(x,t) \in G(B,D) \times (T \setminus \overline{D})$, so by the definition of $K$ and $L$, we must have either $(x,t) \in (X \times T) \setminus K$ or $(x,t) \in (X\times T) \setminus L$. Therefore, $K$ and $L$ are disjoint. 

    For $B,D \in \mathscr{B}$, define 

    \[
    H(B,D) = \{x \in X  \ | \  K_x \subset B \text{ and } L_x \subset D\},
    \]

    where $K_x  = \{t\in T \ | \  (x,t) \in K\}$ and $L_x = \{t \in T| (x,t) \in L\}$. As in Example \ref{coveringexample}, $\mathscr{H}' = \{H(B,D) \ |\  (B,D) \in \mathscr{B}^2\}$ is a pre-$\mathscr{B}$-covering of $X$. Therefore, to show that the $\mathscr{B}$-covering associated with $\mathscr{H}'$ refines $\mathscr{G}$, it remains to show that $H(B,D) \subset G(B,D)$ for all $(B,D) \in \mathcal{S}_{\mathscr{B}}$.

    Let $(B,D) \in \mathcal{S}_{\mathscr{B}}$, and let $x \in H(B,D)$. Then, $K_x \subset B$ and $L_x \subset D$. By definition of $K$ and $L$, we have

    \[
    \{x\} \times (T \setminus B) \subset \bigcup\{G(B',D')\times (T \setminus \overline{B'}) \ | \  B',D' \in \mathscr{B}, \overline{B'} \cap \overline{D'} = \emptyset\},
    \]

    and

    \[
    \{x\} \times (T \setminus D) \subset \bigcup\{G(B'',D'')\times (T \setminus \overline{D''}) \ | \  B'',D'' \in \mathscr{B}, \overline{B''} \cap \overline{D''} = \emptyset\}.
    \]

    Since $\{x\} \times (T\setminus B)$ is compact, it is covered by finitely many elements $\{G(B'_i, D'_i) \times (T \setminus \overline{B_i'})| 1 \leq i \leq n\}$. Therefore, we have $T - B \subset \bigcup_{i=1}^n (T \setminus \overline{B_i'})$ and $x \in \bigcap_{i=1}^n G(B_i, D_i)$. Similarly, we may find a finite family $\{(B''_1, D''_1), \dots, (B''_m, D''_m)\}$ such that $T- D \subset \bigcup_{i=1}^m (T \setminus \overline{D_i''})$ and $x \in \bigcap_{i=1}^m G(B''_i, D''_i)$. Therefore, we have $B \subset \bigcap_{i=1}^n B'_i$, $D \subset \bigcap_{i=1}^m D''_i$, and 
    
    \[ x \in G(\bigcap_{i=1}^n B'_i \cap \bigcap_{i=1}^m B''_i , \bigcap_{i=1}^n D'_i \cap \bigcap_{i=1}^m D''_i) \subset G(B,D).\]

    Thus, we have that $H(B,D) \subset G(B,D)$ for all $(B,D) \in \mathcal{S}_{\mathscr{B}}$. Therefore, the $\mathscr{B}$-covering associated with $\mathscr{H}'$ refines $\mathscr{G}$. 
\end{proof}

    The proof of Theorem \ref{the_theorem} relies on the following characterization of normality of products with a compact factor, which allows us to work with $\mathscr{B}$-coverings instead of disjoint closed sets. The proof of this particular statement of Rudin's Lemma can be found in \cite{Kunen84}. 

\begin{theorem}[Rudin's Lemma]
    \label{characterizationtheorem}\cites{Kunen84, Rudin73, Beslagic91} The product space $X \times T$ is normal if and only if $X$ is normal and every $\mathscr{B}$-covering of $X$ has a locally finite open refinement.
\end{theorem}

\section{Proof of Theorem}

\begin{theorem}\label{the_theorem}
    Suppose that $T$ is compact and that $Y$ is the image of $X$ under a perfect map $p$, $X$ is normal, and $Y \times T$ is normal. Then, $X\times T$ is normal.
\end{theorem}

\begin{proof}

We may assume without loss of generality that $T$ is not a singleton. 

By Rudin's Lemma, it suffices to show that every $\mathscr{B}$-covering of $X$ has a locally finite open refinement. For any $\mathscr{B}$-covering $\mathscr{G}$ of $X$, there exists a pre-$\mathscr{B}$-covering whose associated $\mathscr{B}$-covering refines $\mathscr{G}$ by Lemma \ref{precoveringlemma}, so it is enough to show that for any pre-$\mathscr{B}$-covering of $X$, the associated $\mathscr{B}$-covering has a locally finite open refinement. 

Let $\mathscr{G}'$ be a pre-$\mathscr{B}$-covering of $X$, and let $G: \mathscr{B}^2 \to \mathscr{G}'$ be the map associated with $\mathscr{G}'$. We will show that $\mathscr{G} = \{G(B,D)| (B,D)\in \mathcal{S}_{\mathscr{B}}\}$ has a locally finite open refinement. For each $(B,D) \in \mathscr{B}^2$, define

\[
H(B,D) = \{y \in Y \ | \  p^{-1}\{y\} \subset G(B,D)\}.
\]

It is easy to check that $H$ is intersection preserving. Define ${\mathscr{H}' = \{H(B,D)| (B,D) \in \mathscr{B}^2\}}$; we claim that $\mathscr{H}'$ is an open cover of $Y$. Indeed, for each $y \in Y$, $p^{-1}\{y\}$ can be covered by finitely many elements $\{G(B_1,D_1), \dots, G(B_n,D_n)\}$ of $\mathscr{G}$. Then, since $\mathscr{B}$ is closed under finite union, by Remark \ref{intersectionpropertyremark},

\[
p^{-1}\{y\} \subset \bigcup_{i=1}^n G(B_i,D_i) \subset G(\bigcup_{i=1}^n B_i, \bigcup_{i=1}^n D_i) \in \mathscr{G}',
\]

so $y \in H(\bigcup_{i=1}^n B_i, \bigcup_{i=1}^n D_i)$. Furthermore, since $G(\bigcup_{i=1}^n B_i, \bigcup_{i=1}^n D_i)$ is an open set containing $p^{-1}\{y\}$, from Theorem 1.4.13 in \cite{Engelking89} we know there is some neighbourhood $U$ of $y$ such that $p^{-1}(U) \subset G(\bigcup_{i=1}^n B_i, \bigcup_{i=1}^n D_i)$; therefore $U \subset H(\bigcup_{i=1}^n B_i, \bigcup_{i=1}^n D_i)$ and so $H(\bigcup_{i=1}^n B_i, \bigcup_{i=1}^n D_i)$ is open.

For each $B,D \in \mathscr{B}$, define 

\[
K_B = ((Y \setminus H(B,T)) \times T) \cup (Y \times \overline{B})
\]

and

\[
L_D = ((Y \setminus H(T,D))\times T) \cup (Y \times \overline{D}).
\]

    Define

    \[
    K = \bigcap_{B \in \mathscr{B}} K_{B}, \ \ \ \  \ \ \ \  L = \bigcap_{D \in \mathscr{B}} L_{D}.
    \]

    Evidently $K$ and $L$ are closed. We show they are disjoint; indeed, note that for $B, D \in \mathscr{B}$, we have

    \[
    K_{B} \cap L_{D} = ((Y \setminus (H(B, T) \cup H(T, D)))\times T) \cup ((Y \setminus H(T,D))\times \overline{B}) \cup ((Y \setminus H(B, T))\times \overline{D}) \cup (Y \times (\overline{B}\cap \overline{D})).
    \]

    It suffices to show that each of the terms in the above union are individually empty when intersected over all pairs $(B, D) \in \mathscr{B}^2$. The second and third terms are both empty when this intersection is taken since $\mathscr{H}'$ covers $Y$, and for any $H(B,D) \in \mathscr{H}'$, we have $H(B,D) \subset H(B,T)$ and $H(B,D) \subset H(T,D)$. Since $T$ is normal and contains at least two distinct points, there are basis elements $B, D \in \mathscr{B}$ such that $\overline{B} \cap \overline{D} = \emptyset$, so the fourth term is empty under this intersection. Thus, we have

    \[
        K \cap L = \bigcap_{(B, D) \in \mathscr{B}^2} K_{B} \cap L_{D} = ((Y \setminus \bigcup_{(B, D) \in \mathscr{B}^2}(H(B,T) \cup H(T, D))) \times T).
    \]

    Since $\mathscr{H}'$ covers $Y$, we have 
    \[\bigcup_{(B, D) \in \mathscr{B}^2} H(B,T) \cup H(T, D) \supset \bigcup_{(B, D) \in \mathscr{B}^2} H(B,T) \cap H(T, D) = \bigcup_{(B, D) \in \mathscr{B}^2} H(B,D) = Y,\]
    
    and therefore $K\cap L = \emptyset$.

    We claim that $\mathscr{H} = \{H(B,D) \ | \  (B,D)\in \mathcal{S}_{\mathscr{B}}\} \subset \mathscr{H}'$ is a $\mathscr{B}$-covering of $Y$. It remains to show that $\mathscr{H}$ covers $Y$. Let $y \in Y$. As shown in Example \ref{coveringexample} $K_y = \{t\in T \ |\  (y,t) \in K\}$ and $L_y = \{t \in T \ |\  (y,t) \in L\}$ are closed disjoint subsets of $T$, so there exist $B,D \in \mathscr{B}$ with $\overline{B} \cap \overline{D} = \emptyset$ such that $K_y \subset B$ and $L_y \subset D$ by the fact that $\mathscr{B}$ is closed under finite unions. We show that $y \in H(B,T)$, an analogous argument shows that $y \in H(T,D)$ and so $y \in H(B,T) \cap H(T,D) = H(B,D)$.

    Define $\mathcal{J} = \{B' \in \mathscr{B} \ |\  y \in H(B', T)\}$, we claim that there is some $B' \in \mathcal{J}$ such that $B' \subset B$, and thus $y \in H(B',T) \subset H(B,T)$. By definition of $K$, we must have $\bigcap_{B' \in \mathcal{J}} \overline{B'} \subset K_y \subset B$. Therefore, $\{T \setminus \overline{B'} \ |\  B' \in \mathcal{J}\}$ is an open cover of $T \setminus B$, so there exists a finite subcover $\{T \setminus \overline{B'_1}, \dots, T\setminus \overline{B'_n}\}$. In particular, we have $\bigcap_{i=1}^n \overline{B'_i} \subset B$. But $y \in H(B'_i,T)$ for each $1 \leq i \leq n$, and therefore 

    \[
    y \in \bigcap_{i=1}^n H(B'_i,T) \subset H (\bigcap_{i=1}^n B'_i, T) 
    \]

    so $B'' = \bigcap_{i=1}^n B'_i \in  \mathcal{J}$. But since $B'' \subset B$, we have

    \[
    y \in H(B'', T) \subset H(B,T), 
    \]

    and so $y \in H(B,T)$. Therefore, for each $y \in Y$, there exist $(B,D) \in \mathcal{S}_{\mathscr{B}}$ such that $y \in H(B,D)$.

    Since $\mathscr{H}$ is a $\mathscr{B}$-covering of $Y$, by Rudin's Lemma, it has a locally finite open refinement $\mathscr{V}$. For each $V \in \mathscr{V}$, there exists $(B_V, D_V) \in \mathcal{S}_{\mathscr{B}}$ such that $V \subset H(B_V, D_V)$, so by definition of $\mathscr{H}'$, we have that for each $V \in \mathscr{V}$,
    
    \[
    p^{-1}(V) \subset p^{-1}(H(B_V,D_V)) \subset G(B_V,D_V).
    \]
    
     Therefore, the set $\{p^{-1}(V) \ |\  V \in \mathscr{V}\}$ is a locally finite open refinement of $\mathscr{G}$. Thus, by Rudin's Lemma, $X \times T$ is normal. 
   
\end{proof}

This allows us to answer another problem in \cite{Kunen84}. 

\begin{corollary}\label{the_corollary}
    Suppose $T$ is compact, and $X \times T$ is normal. Then, $X \times T^n$ is normal for each $n \in \mathbb{N}$.
\end{corollary}

\begin{proof}
    This follows by induction on $n$, noticing that the projection $p: X \times T \to X$ is perfect because $T$ is compact. 
\end{proof}

Corollary \ref{the_corollary} does not hold when $n$ is replaced with $\omega$. To see why, we need the following lemma.

\begin{lemma}\label{Morita_lemma}
    \cite{Morita61} Let $Q$ and $Q'$ be any two compact Hausdorff spaces. If $Q'$ is either a closed subset or a continuous image of $Q$, then if $X \times Q$ is normal, so is $X \times Q'$.
\end{lemma}

\begin{proposition}
    There exists a space $X$ and a compact set $T$ such that $X \times T$ is normal but $X \times T^{\omega}$ is not normal. 
\end{proposition}

\begin{proof}
    Let $T = 2$ with the discrete topology, and let $X$ be a Dowker space, that is, a normal space for which $X \times [0,1]$ is not normal \cite{Rudin71}. Since $T$ is discrete and $X$ is normal, we know that $X \times T$ is normal. Suppose that $X \times 2^{\omega}$ is normal. Since $2^{\omega}$ is the Cantor set, $[0,1]$ is the continuous image of $2^{\omega}$ (see \cite{Engelking89}), so by Lemma \ref{Morita_lemma}, $X \times [0,1]$ is normal, contradicting the fact that $X$ is Dowker.  
\end{proof}

\section*{Acknowledgements}

The author would like to thank Clovis Hamel and Franklin D. Tall for their invaluable comments, corrections, and references.

\end{document}